\documentclass[12pt, a4paper]{amsart} 

\usepackage[utf8]{inputenc}

\usepackage{amsthm}
\usepackage{euscript}
\usepackage[usenames]{color}
\usepackage{amsmath}
\usepackage{amsfonts}
\usepackage{amssymb}
\usepackage{graphicx}
\usepackage[all]{xy}
\usepackage{amscd}
\usepackage{IEEEtrantools}
\usepackage{stmaryrd}
\usepackage{commath}
\usepackage{tikz}
\usepackage{todonotes}
\usepackage[colorlinks=true,linktocpage=true,pagebackref=true, citecolor=red,linkcolor=blue]{hyperref}
\usepackage[shortlabels]{enumitem}
\usepackage{caption}
\usepackage{subfigure}
\usepackage{hyperref}
\usepackage{float}
\setlist{topsep=-1pt}

\makeatletter
\newcommand{\hlabel}[1]{%
   \protected@write \@auxout {}%
         {\string \newlabel {#1}{{\theteorema}{\thepage}{}{#1}{}} }%
   \hypertarget{#1}{}
}
\makeatother

\newtheorem{tma}{Theorem}[section]
\theoremstyle{definition} \newtheorem{defi}[tma]{Definition}
\theoremstyle{definition} 
\newtheorem{prop}[tma]{Proposition}

\newtheorem{lema}[tma]{Lemma}
\theoremstyle{definition} 
\newtheorem{coro}[tma]{Corollary}
\theoremstyle{definition} \newtheorem{ejem}[tma]{Example}
\theoremstyle{definition} 

\newtheorem{remark}[tma]{Remark}

\newcommand{\C}{\mathbb{C}}

\newcommand{\PGL}{\mathrm{PGL}}

\title{Configuration spaces of orbits and their $S_n$-equivariant $E$-polynomials}\author{Alejandro Calleja}
\address{Instituto de Ciencias Matemáticas, Consejo Superior de Investigaciones Científicas, Nicolás Cabrera, 15, Madrid, 28049, Madrid, Spain}

\address{Departmento de Álgebra, Geometría y Topología, Facultad de Ciencias Matemáticas, Universidad Complutense de Madrid, Plaza de las Ciencias, 3, Madrid, 28040, Madrid, Spain}

\email{alejac03@ucm.es}


\begin{document}
\maketitle

\begin{abstract}
    In this paper, we study the configuration space of orbits, a generalization of the configuration space of points but for algebraic varieties that are acted by an algebraic reductive group. The main objective of this work is to study the $E$-polynomials of these spaces and their quotients by $S_n$. For this purpose, we develop a novel method for computing the $S_n$-equivariant $E$-polynomial of an algebraic variety, and we apply it to this kind of varieties.
\end{abstract}

\section{Introduction}
Given a space $X$ the \emph{$n$-th ordered configuration space of $X$} is defined as the set
\begin{equation*}
    C_n(X)=\{(x_1,\ldots,x_n)\in X^n \ | \ x_i\neq x_j \text{ if } i\neq j\}.
\end{equation*}
By quotienting this space by the permutation action of the symmetric group $S_n$, one can consider the \emph{$n$-th unordered configuration space of $X$} $B_n(X)=C_n(X)/S_n$.

The cohomology of these spaces has been widely studied in the case in which $X$ is a smooth manifold, especially for surfaces \cite{arnold2014cohomology,bodigheimer2006rational,drummond2016betti,knudsen2017betti, napolitano2003cohomology,salvatore2004configuration,wang2002braid}, but also in arbitrary dimension, with the works of Fulton and MacPherson \cite{fulton1994compactification}, Kriz \cite{kriz1994rational}, Torato \cite{totaro1996configuration} and Cohen and Taylor \cite{cohen2006computations}. Nevertheless, in the case in which $X$ is a possibly singular algebraic variety very little is known. In this setting, the cohomology of the varieties is enriched with an additional structure known as a mixed Hodge structure. The aim of this work is to study this structure for configuration spaces of algebraic varieties, by finding a method for computing their $E$-polynomials in terms of the one of the algebraic variety $X$. Recall that the $E$-polynomial of a variety $Z$ is given by the formula
\begin{equation*}
    e(Z)=\sum_{k,p,q}(-1)^kh_c^{k,p,q}(Z)\,u^pv^q \in \mathbb{Z}[u,v],
\end{equation*}
where the $h_c^{k,p,q}(Z)$ are the Hodge numbers for the compactly supported cohomology of $Z$.

Moreover, when the algebraic variety $X$ is acted by an algebraic group $G$, a generalization of the configuration space can be considered, the \emph{configuration space of orbits} (see Definition \ref{conforbits}). This space consists of tuples of $n$ points of $X$ whose orbits through the action of $G$ are pairwise disjoint. This kind of spaces were firstly introduced in \cite{merino1997orbit} to study universal covers of configuration spaces of points and classifying spaces of normal subgroups of surface braid groups. The cohomology of these spaces was studied in \cite{bibby2022combinatorics} and \cite{bibby2023generating} for the case were $G$ is a finite group; for which the authors give a description of the spaces in terms of some combinatorial objects called Dowling posets. Moreover, they compute their $E$-polynomial by relating it with the characteristic polynomial of the mentioned posets. Nevertheless, in the case of non-finite groups nothing is known. In this paper, we address the case where $G$ is a complex connected algebraic reductive group, computing the $E$-polynomial of this space and of its quotient by $S_n$ by reducing this problem to the one of computing the $E$-polynomials of the usual ordered and unordered configuration spaces of the Geometric Invariant Theory (GIT) quotient $X\sslash G$. To be precise, in Section \ref{confspace} we prove the following theorem, which is a combination of Corolary \ref{orbitsthroughquot} and Proposition \ref{eqvarofCnXGquot}:

\begin{tma}\label{maintma}
    Let $X$ be a complex algebraic variety and $G$ a complex algebraic connected reductive group that acts freely on $X$. The $S_n$-equivariant $E$-polynomial of the configuration space of orbits is given by the formula
    $$e_{S_{n}}\left(C_n(X,G)\right)=e_{S_{n}}\left(C_n(X\sslash G^n)\right)\otimes e_{S_{n}}(G^n)$$
\end{tma}

In Section \ref{pmformula} we reduce the problem of computing the $S_n$-equivariant $E$-polynomial of the ordered configuration space $C_n(X\sslash G^n)$ to computing the $E$-polynomials of the quotients $C_n(Z\sslash G^n)/S_n$ for an arbitrary algebraic variety $Z$. For this purpose, we prove in Theorem \ref{CnquotSn} the formula
$$e\left(C_n(Z)/S_n\right)=e\left(\mathrm{Sym}^n(Z)\right)-\sum_{l=1}^{n-1}p(n,l)e\left(C_l(Z)/S_l\right),$$
where $p(n,l)$ is the number of partitions of $n$ of length $l$ and $\mathrm{Sym}^n(X) = X^n/S_n$ is the $n$-th symmetric product of $X$. This formula allows us to compute the $E$-polynomial of $C_n(Z)/S_n$ in a recursive way.

Even though we focus on the $E$-polynomial, most of the results are also valid for motives in the Grothendieck ring of algebraic varieties, thanks to Vogel's PhD. Thesis \cite{vogelthesis}, where he studies equivariant motives. We recall that the Grothendieck ring of algebraic varieties is defined as the formal ring $\mathbf{K}Var$ generated by the isomorphism classes of algebraic varieties modulo the relations $[X]=[U]+[Y]$ whenever we have a stratification $X=U\sqcup Y$ with $U$ open and $Y$ closed, and with the product being defined by $[X][Y]=[X\times Y]$. The main inconvenience for making all the work with motives is Theorem \ref{fibrationCn}, since the fiber bundle introduced there is not algebraic, so we do not have multiplicativity, as we have with the $E$-polynomial.

To avoid misunderstandings with the configuration space of orbits, from now on we will refer to the usual ordered configuration space as the \emph{configuration space of points}.  

A very interesting application of this kind of configuration spaces can be appreciated when studying the applications of character varieties to knot theory. Let us recall the definition of a character variety. Let $\Gamma$ be a finitely generated group and let $G$ be a complex reductive group. It is an easy fact that the set $R(\Gamma,G)$ of all representations $\rho:\Gamma\to G$ of $\Gamma$ into $G$ has a structure of algebraic variety, called the \emph{representation variety} of $\Gamma$ into $G$. Moreover, since $G$ acts on $R(\Gamma,G)$ by conjugation, we can consider the GIT quotient $\mathfrak{X}(\Gamma,G)=R(\Gamma,G)\sslash G$, called the \emph{character variety} of $\Gamma$ into $G$. 

This kind of varieties has been widely studied in recent years, mainly due to their relation with the non-abelian Hodge correspondence \cite{corlette1988flat,hitchin1987self,simpson1994moduli}. This correspondence relates the moduli space of Higgs bundles over an orientable surface $\Sigma$ of genus $g$ with the character variety $\mathfrak{X}(\pi_1(\Sigma,G))$ and with the moduli of flat connections over the surface.

However, in the last few years, a new interest in these varieties has arisen due to their connection with knot theory. Let $K \subset \mathbb{R}^3$ be a knot. If we take the group $\Gamma_K=\pi_1(\mathbb{R}^3-K)$, we can consider the character variety $\mathfrak{X}(\Gamma_K,G)$. Then, we can define a knot invariant by assigning to $K$ the $E$-polynomial of its character variety $\mathfrak{X}(\Gamma_K,G)$. Following this direction, trivial links were studied in \cite{florentino2012singularities,florentino2021serre,lawton2008minimal,lawton2016polynomial}, the figure eight knot in \cite{falbel2016character,heusener2016sl}, and torus knots in \cite{GONZALEZPRIETO2022852,kitano2012twisted,martínezmartínez2012su2character,munoz2016geometry,muñoz2009sl2ccharacter}. This last family of knots is the most studied one so far but, due to the great complexity of the computations, it has only been studied for a few algebraic groups, and all of them are of low dimension. In this sense, the best result reported in the literature is due to González-Prieto and Muñoz in \cite{GONZALEZPRIETO2022852}, where they were able to compute the $E$-polynomial of $\mathfrak{X}(\Gamma_K,\mathrm{SL}_4(\mathbb{C}))$ for torus knots in a recursive way. 

The ideas used in that work by González-Prieto and Muñoz are very general, but if we want to completely generalize them to compute the $E$-polynomial of $\mathfrak{X}\left(\Gamma_K,\mathrm{SL}_r(\mathbb{C})\right)$ for arbitrary range $r$, then we must understand the configuration space $C_n\left(R^{\mathrm{irr}}(\Gamma_K,\mathrm{PGL}_r(\mathbb{C})\right)$ whose elements are tuples of pairwise non-isomorphic irreducible representations of $\Gamma_K$ into $\mathrm{SL}_r(\mathbb{C})$. Since two representations are isomorphic if and only if they are conjugated by an element of $\mathrm{PGL}_r(\mathbb{C})$, we can say that the variety $C_n\left(R^{\mathrm{irr}}(\Gamma_K,\mathrm{PGL}_r(\mathbb{C}))\right)$ is the set of $n$-tuples $(x_1,\ldots,x_n)$ of elements of the variety $R^{\mathrm{irr}}\left(\Gamma_K,\mathrm{SL}_r(\mathbb{C})\right)$ of irreducible representations, such that $x_i$ is not in the orbit through the action of $\mathrm{PGL}_r(\mathbb{C})$ of $x_j$ if $i\neq j$. In other words, this variety is not other but the configuration space of orbits of $R^{\mathrm{irr}}\left(\Gamma_K,\mathrm{SL}_r(\mathbb{C})\right)$ for the action of $\mathrm{PGL}_r$.

This work is structured as follows. Section \ref{sec:Hodge} is devoted to introducing the main concepts of Hodge Theory and Geometric Invariant Theory with which we are going to work. We also study the relation between these two branches. In Section \ref{confspace} we introduce the configuration space of orbits, and study their relation with the configuration space of points.  In section \ref{CnSn} we compute the $E$-polynomials of $C_n(X)$ and $C_n(X)/S_n$ in terms of the $E$-polynomial of $X$. Section \ref{pmformula} is devoted to finding a method for computing the $S_n$ equivariant $E$-polynomial of an arbitrary algebraic variety, which will allow us to compute the $E$-polynomial of the quotient by $S_n$ of the total space of a fiber bundle in terms of the $E$-polynomial of the quotients basis and of the fiber. Finally, in Section \ref{CnSlambda} we apply the method of Section \ref{pmformula} to the case of configuration spaces of points, allowing us to compute the $E$-polynomial of the quotient of the configuration space of orbits by $S_n$.

\textbf{Aknowledgements:} I want to thank my PhD advisor Ángel González-Prieto for his invaluable help and support without which this paper would have never been possible. I also thank Prof. Carlos Florentino for many useful comments.

The author has been partially supported by Ministerio de Ciencia e Innovación Projects CEX-2019-000904-S-20-5 and PID2021-124440NB-I00 (Spain).

\section{Hodge Theory and GIT quotients}\label{sec:Hodge}
In this section we will review the main concepts of Hodge Theory and Geometric Invariant Theory, studying the relation between them.
\subsection{The E-polynomial} Let $H$ be a complex vector space. A \emph{pure Hodge structure of weight $k$ on $H$} is a decomposition $H=\bigoplus_{p+q=k}H^{p,q}$ such that $\overline{H^{p,q}}=H^{q,p}$. If we consider the descending filtration $F^p=\bigoplus_{s\geq p}H^{s,k-s}$, we can define the associate graded complex $\mathrm{Gr}_F^p(H)=F^p/F^{p+1}=H^{p,k-p}$.

It is well known that the cohomology groups $H^k(Z;\mathbb{C})$ of a compact K\"ahler manifold $Z$ admit a pure Hodge structure of weight $k$. Deligne proved in \cite{deligne1971theorie} that this could be extended for general algebraic varieties. In more detail, he proved that the complex cohomology of these varieties has a descending Hodge filtration
\begin{equation*}
    H^k(Z; \mathbb{C}) = F_{0} \supset \cdots \supset F_{\ell} \supset F_{\ell+1} \supset \cdots \supset 0,
\end{equation*}
and the rational cohomology has an ascending \emph{weigh filtration} 
\begin{equation*}
    0= W_{0}\subset \cdots \subset W_{\ell} \subset W_{\ell+1}\subset \cdots \subset H^k(Z;\mathbb{Q})
\end{equation*}
such that $F$ induces a pure hodge structure of weight $k$ on each $\mathrm{Gr}_{k}^{W_{\mathbb{C}}}H^k(Z;\mathbb{C})$, where $W_{\mathbb{C}}$ is the complexified weight filtration. With these filtrations, we define the spaces
\begin{equation*}
    H^{k,p,q}(Z)=\mathrm{Gr}_F^p\mathrm{Gr}_{p+q}^{W_{\mathbb{C}}}H^k(Z;\mathbb{C})
\end{equation*}
and the Hodge numbers $h^{k,p,q}(Z)=\dim H^{k,p,q}(Z)$.

We can repeat these constructions on the compactly supported cohomology, obtaining the spaces $H^{k,p,q}_c(Z)$ and the Hodge numbers $h_c^{k,p,q}(Z)=\dim H_c^{k,p,q}(Z)$. The $E$-polynomial of $Z$ is defined as
\begin{equation*}
    e(Z)=\sum_{k,p,q}(-1)^kh^{k,p,q}_c(Z)\, u^pv^q.
\end{equation*}
When $h_c^{k,p,q}(Z)=0$ for $p\neq q$ we say that $Z$ is \emph{of balanced type}. Observe that in this case, the $E$-polynomial only depends on the product $uv$, so it is customary to perform the change of variables $q=uv$.

The main objective of this paper is to compute the $E$-polynomial for a special kind of variety, the configuration space of orbits $C_n(X,G)$, but also of its quotient $C_n(X,G)/S_n$ by the action of the symmetric group $S_n$. For the latter, we will need the equivariant version of the $E$-polynomial.

In this direction, let $H$ be a finite group that acts on an algebraic variety $Z$. This action induces a new action of $H$ in the cohomology $H_c^k(Z)$ that respects its mixed Hodge structure. Therefore, each space $H_c^{k,p,q}(Z)$ is a representation of $H$. Let $[H_c^{k,p,q}(Z)]\in R(H)$ be the class of $H_c^{k,p,q}(Z)$ in the representation ring of $H$. 
\begin{defi}
    We define the \emph{equivariant $E$-polynomial} of $Z$ as 
    \begin{equation*}
        e_H(Z)=\sum_{k,p,q}(-1)^k[H_c^{k,p,q}(Z)]\, u^pv^q\in R(H)[u,v].
    \end{equation*}
\end{defi}

The interesting fact about the equivariant $E$-polynomial is that it satisfies the following properties:

\begin{prop}\label{propeqvar}
    \begin{itemize}
        \item[(a)] For every subgroup $K<H$ it holds
        \begin{equation*}
            e(Z/K)=\sum_{k,p,q}(-1)^k\dim\left(H_c^{k,p,q}(Z)^Ku^pv^q\right).
        \end{equation*}
        where for a given $K$-module $V$, we denote $V^K$ for the subspace of $V$ invariant under the action of $K$.

        \item[(b)] If $Y\subset Z$ is a closed subvariety of $Z$ and $U=Z-Y$, then $e_H(Z)=e_H(Y)+e_H(U)$.
        
        \item[(c)] Let $F\to E\to B$ be a fibration locally trivial in the analytic topology such that $F$ and $B$ are smooth. If $\pi_1(B)$ acts trivially on the cohomology of $F$, then
        \begin{equation*}
            e_H(E)=e_H(F)\otimes e_H(B).
        \end{equation*}
    \end{itemize}
\end{prop}

\begin{proof}
    (a) follows from the fact that for every finite group $K$, it holds $H_c^{k,p,q}(X/K)=H_c^{k,p,q}(X)^K$ (see \cite{dimca1997purity} or \cite{florentino2021serre}). A proof of (b) can be found in \cite{peters2008mixed}. For part (c) we again refer to \cite{dimca1997purity} or \cite{florentino2021serre}.
\end{proof}

\begin{remark}\label{coeftrivial}
\begin{enumerate}
    \item Observe that taking $K=\{e\}$ in part (a) of the proposition, we recover the usual $E$-polynomial, and if we take $K=H$ we just get the $E$-polynomial of the quotient $Z/H$. In fact, the latter is just the coefficient of the trivial representation of $H$ in $e_H(Z)$.
    \item Using the first part of this remark we can see that parts (b) and (c) of Proposition \ref{propeqvar} also hold for the usual $E$-polynomial.
    \item If the fibration $E\to B$ is locally trivial in the Zariski topology, then the monodromy action of $\pi_1(B)$ on the cohomology of the fiber is automatically trivial. The same holds if $E \to B$ is a principal $G$-bundle with $G$ a connected group.
    \end{enumerate}  
\end{remark}

\subsection{Geometric invariant theory}
Let $Z=\mathrm{Spec}(R)$ be an affine complex variety, where $R$ is a finitely generated $\mathbb{C}$-algebra. Let $G$ be a complex algebraic reductive group that acts on $Z$. Recall that the GIT quotient of $Z$ by $G$ is given as 
\begin{equation*}
Z\sslash G=\mathrm{Spec}(R^G),
\end{equation*}
where $R^G$ is the algebra of $G$-invariant elements of $R$.

In the general case of a quasi-projective algebraic variety, we will suppose that the action of $G$ can be extended to linear action on an ample line bundle $L$ over $Z$ (i.e., the action can be linearized). This defines the subsets $Z^S(L)$ and $Z^{SS}(L)$ of stable and semi-stable points of $Z$. The main result in Geometric Invariant Theory states that, in this case, there exists a \emph{good quotient} $(Z\sslash G,\pi)$ of $Z^{SS}(L)$ (see \cite{mumford1994geometric}). We call the variety $Z\sslash G$ the GIT quotient of $Z$. For simplicity, from now on we will suppose that $Z=Z^{SS}(L)$ (which is true, for example, when the variety is affine). In other case, all the results remain true, just by substituting $Z$ with $Z^{SS}(L)$. 

We recall some properties of GIT quotients that will be useful throughout this paper: 

\begin{tma}\label{luna}
    Let $Z$ be a variety with an action of a reductive group $G$. Then, the GIT quotient $ Z\sslash G$ satisfies the following properties: \begin{itemize}
    \item It is a categorical quotient, that is, for every variety $Y$ and every $G$-equivariant regular morphism $Z\to Y$, there exists a unique regular morphism $Z\sslash G\to Y$ such that the diagram
    \begin{equation*}
        \xymatrix{
        Z\ar[d]\ar[rd]&\\
        Z\sslash G\ar[r]&Y
        }
    \end{equation*}
    conmutes.
    \item For every two closed subsets $W_1, W_2\subset Z$, we have $\overline{\pi(W_1)}\cap \overline{\pi(W_2)}=\varnothing$ if and only if $W_1 \cap W_2 = \varnothing$.
    \item If the action of $G$ is free, then the GIT quotient is a principal $G$-bundle in the analytic topology.
    \end{itemize}
\begin{proof}
The first two statements can be found in \cite{newstead1978introduction}. The last one is a direct corollary of the renowned Luna slice theorem (see \cite{luna1972orbites} or \cite[Proposition 5.7]{drezet2004luna}).
\end{proof}
\end{tma}

\subsection{Equivariant $E$-polynomial of a pseudo-quotient}\label{pseudoquotients}
We are now going to study the relation between Hodge theory and Geometric Invariant Theory. As an easy consequence  of Proposition \ref{propeqvar} and Luna slice theorem \ref{luna}, we get the following:
\begin{tma}\label{eqvarGIT}
    Let $Z$ be an algebraic variety and let $G$ be an algebraic reductive group that acts freely on $Z$. Let $H$ be a finite group that acts on $Z$. Then 
    \begin{equation*}
        e_H(Z)=e_H(Z\sslash G)\otimes e_H(G).
    \end{equation*}
\end{tma}

In this paper, we will need a more powerful result than Theorem \ref{eqvarGIT}. We will need to study a generalization of GIT quotients, the pseudo-quotients. Let us recall the definition of a pseudo-quotient from \cite{Gonz_lez_Prieto_2023}:

\begin{defi}
    Let $X$ be an algebraic variety and let $G$ be an algebraic group that acts on $X$. A \emph{pseudo-quotient} for the action of $G$ on $X$ is a surjective $G$-invariant regular morphism $\pi:X\to Y$ such that, for any disjoint $G$-invariant closed sets $W_1,W_2\subset X$, we have that $\overline{\pi(W_1)}\cap\overline{\pi(W_2)}=\varnothing$.
\end{defi}

In \cite{Gonz_lez_Prieto_2023}, it was proven that if $\pi:X\to Y=X\sslash G$ is the GIT quotient and $\pi':X\to Y'$ is a pseudo-quotient for the action of $H$, then there is a regular bijective map $\alpha:Y\to Y'$. Furthermore, if $Y'$ is normal, then $\alpha$ is an isomorphism. We are now going to extend this result to the case in which all the varieties $X,Y$ and $Y'$ have an action of a finite group $H$ and the maps $\pi$ and $\pi'$ are $H$-equivariant, proving that the resulting isomorphism is also $H$-equivariant. This will give us an equality of $H$-equivariant $E$-polynomials:
$$e_H(Y)=e_H(Y').$$

\begin{lema}\label{equiviso}
    Let $\pi:X\to Y = X \sslash G$ be the GIT quotient and $\pi':X\to Y'$ be a pseudo-quotient with $Y$ normal. Let $H$ be a finite group acting on $X, Y$ and $Y'$. Suppose that $\pi$ and $\pi'$ are $H$-equivariant maps. Then the isomorphism $\alpha:Y\to Y'$ is also $H$-equivariant.
\end{lema}

\begin{proof}
    Let $h\in H$. As $\pi$ and $\pi'$ are $H$-equivariant, we have that for all $x\in X$, $\pi(h(x))=h(\pi(x))$ and $\pi'(h(x))=h(\pi'(x))$. Let $y\in Y$. Then $y=\pi(x)$ for some $x\in X$. Therefore, by the commutativity of the diagram
    \begin{equation*}
        \xymatrix{
        &X\ar[ld]_-{\pi}\ar[rd]^-{\pi'}&\\
        Y\ar[rr]^-{\alpha}&&Y
        }
    \end{equation*}
    we get
    \begin{eqnarray*} \alpha(h(y))&=\alpha\left(h\left(\pi(x)\right)\right)=\alpha\left(\pi\left(h(x)\right)\right)=\pi'\left(h(x)\right)\\
    &=h\left(\pi'(x)\right)=h\left(\alpha\left(\pi(x)\right)\right)=h\left(\alpha(y)\right),
    \end{eqnarray*}
    so $\alpha$ is $H$-equivariant.
\end{proof}

This result in particular implies that $e_H(Y)=e_H(Y')$ when $Y'$ is normal. The case when $Y'$ is not normal is covered in the following.

\begin{prop}\label{epolpseudo}
    Let $\pi:X\to Y = X \sslash G$ be the GIT quotient and $\pi':X\to Y'$ be a pseudo-quotient. Let $H$ be a finite group that acts on $X$ and consider its action on $Y$ and $Y'$ induced by $\pi$ and $\pi'$ respectively. Then $e_H(Y)=e_H(Y')$.
\end{prop}

\begin{proof}
    Let $Y_1'\subset Y'$ be the set of normal points of $Y'$. Since $\pi'$ is $H$ equivariant, the action of $H$ on $X$ restricts to an action on $X_1:=(\pi')^{-1}(Y_1')$. Therefore, if we call $Y_1=\alpha^{-1}(Y_1')$, then the map $\pi|_{X_1}:X_1\to Y_1$ is a GIT quotient and $\pi'|_{X_1}:X_1\to Y_1'$ is a pseudo-quotient, where $Y_1'$ is normal, so $e_H(Y_1)=e_H(Y_1')$.

    Iterating this process with $X-X_1$, we obtain our varieties stratified as $X=X_1\sqcup X_2\sqcup \ldots \sqcup X_k, Y=Y_1\sqcup Y_2\sqcup \ldots \sqcup Y_k$ and $Y'=Y'_1\sqcup Y'_2\sqcup \ldots \sqcup Y'_k$, with each restriction $\pi|_{X_i}:X_i\to Y_i$ being a GIT quotient and $\pi'|_{X_i}:X_i\to Y_i$ a pseudo-quotient, and therefore, $e_H(Y_i)=e_H(Y_i')$, so finally we get
    \begin{equation*}
        e_H(Y)=\sum_{i=1}^k e_H(Y_i)=\sum_{i=1}^ke_H(Y_i')=e_H(Y').
    \end{equation*}
\end{proof}

\section{Configuration space of orbits}\label{confspace}

Let $X$ be an algebraic variety and $G$ an algebraic connected reductive group that acts over it. Let $\pi:X\to X\sslash G$ be the GIT quotient.

\begin{defi} \label{conforbits}
    We define the \emph{$n$-th configuration space of orbits of $X$ for the action of $G$} as the space
    \begin{equation*}
        C_n(X,G):=\{(x_1,\ldots, x_n)\in X^n \ | \ \pi(x_i)\neq \pi(x_j) \ \text{if} \ i\neq j\}.
    \end{equation*}
\end{defi}

\begin{remark}
    This definition is slightly different to the one given in \cite{merino1997orbit}, where condition $\pi(x_i)\neq \pi(x_j)$ was substituted by $Gx_i\cap Gx_j=\varnothing$. We made this modification to ensure that the map $\pi':C_n(X,G)\to C_n(X\sslash G), (x_1,\ldots,x_n)\mapsto ([x_1],\ldots,[x_n])$ is a pseudo-quotient.
\end{remark}

Our main objective will be to compute the $E$-polynomial of the quotient $C_n(X,G)/\/S_{n}$ in terms of the $E$-polynomial of $X$, where $S_n$ acts by permutation of the coordinates.

In view of Proposition \ref{propeqvar}, for computing the $E$-polynomial of $C_n(X,G)/S_{n}$ it suffices to find a suitable fibration $F\to C_n(X,G)\to B$ such that we can compute the equivariant $E$-polynomial of the fiber and of the base. This can be solved if we take into account the obvious action of $G^n$ on $C_n(X,G)$ induced by the action of $G$ on each coordinate. More precisely, we have the following result:

\begin{prop}\label{mainbundle}
    Let $X$ be a complex algebraic variety and $G$ a complex algebraic connected reductive group that acts freely on $X$. Then the quotient $\pi:C_n(X,G)\to C_n(X,G)\sslash G^n$ is a locally trivial fibration in the analytic topology with fiber $G^n$ and trivial monodromy. 
\end{prop}

\begin{proof}
    The fact that $\pi$ is a fiber bundle in the analytic topology is just a consequence of the Luna Slice Theorem (\ref{luna}).
    The triviality of the monodromy follows from the connectedness of $G^n$.
\end{proof}

\begin{remark}\label{quotisSnequiv}
    Observe that that the action of $S_n$ in $C_n(X,G)$ can be descended to an action on $C_n(X,G)\sslash G$ such that the GIT quotient $\pi:C_n(X,G)\to C_n(X,G)\sslash G^n$ is $S_n$-equivariant. Just statie $\sigma\cdot \pi(x_1,\ldots,x_n):=\pi(\sigma(x_1,\ldots,x_n))$ for any permutation $\sigma\in S_n$, which is well defined since for any $(g_1,\ldots,g_n)\in G^n$, we have $$\sigma((g_1,\ldots,g_n)\cdot(x_1,\ldots,x_n))=\sigma(g_1,\ldots,g_n)\cdot\sigma(x_1,\ldots,x_n).$$
\end{remark}

\begin{coro}\label{orbitsthroughquot}
    In the conditions of Proposition \ref{mainbundle}, we have
    \begin{equation*}
        e_{S_{n}}\left(C_n(X,G)\right)=e_{S_{n}}\left(C_n(X,G)\sslash G^n\right)\otimes e_{S_{n}}(G^n)
    \end{equation*}
\end{coro}

Therefore, if we want to compute the equivariant $E$-polynomial of $C_n(X,G)$, we just need to compute the equivariant $E$-polynomials of the right side of the equation of the previous corollary. It is obvious that the equivariant $E$-polynomial of $G^n$ will be much easier to compute than the one of $C_n(X,G)$, but we still have the problem of computing $e_{S_{n}}(C_n(X,G))\sslash G^n$. However, the following result simplifies the situation.

\begin{prop}\label{eqvarofCnXGquot}
    Let $H$ be a finite group that acts over $C_n(X,G),C_n(X,G)\sslash G^n$ and on $C_n(X\sslash G)$ in such a way that both $\pi:C_n(X,G)\to C_n(X,G)\sslash G^n$ and $\pi':C_n(X,G)\to C_n(X\sslash G)$ are $H$-equivariant. Then,
    \begin{equation*}
        e_H(C_n(X,G)\sslash G^n)=e_H(C_n(X\sslash G)).
    \end{equation*}
\end{prop}

\begin{proof}
    This follows from the obvious fact that $\pi':C_n(X,G)\to C_n(X\sslash G), (x_1,\ldots,x_n)\mapsto ([x_1],\ldots,[x_n])$ is a pseudo-quotient, together with Proposition \ref{epolpseudo}.
\end{proof}

In the case that $H=S_n$ and the action is of permutation of coordinates, it is straightforward that the map $\pi':C_n(X,G)\to C_n(X,G)\sslash G^n$ is $S_n$-equivariant. Combining this with Remark \ref{quotisSnequiv} and Proposition \ref{eqvarofCnXGquot}, we can reduce the problem of computing $e_{S_{n}}(C_n(X,G))$ to the calculation of $e_{S_{n}}(C_n(X\sslash G))$, which is more addressable. In fact, in Section \ref{pmformula}, we will give a method valid for computing both $e_{S_{n}}(G^n)$ and $e_{S_{n}}(C_n(X)\sslash G)$.

\section{$E$-polynomial of $C_n(X)/S_n$}\label{CnSn}

In this section, we will compute the $E$-polynomial of the quotient $C_n(X)/S_n$ in terms of the $E$-polynomial of $X$. For this, we will first need to compute the $E$-polynomial of $C_n(X)$, which was done in \cite{bibby2022combinatorics}, but we will give two new methods for computing it. The reason for doing this is that the techniques that we develop here will be of great interest later in this paper. Let us briefly summarize on what these methods consist:

The first one has a more geometric flavor and consists of computing the $E$ polynomial by using suitable fiber bundles, and will give us the same formula as in \cite{bibby2022combinatorics}. This method will be very useful in the following sections. In concrete, Theorem \ref{fibrationCn} will allow us to reduce the problem of computing $E$-polynomials of the quotients $C_n(X)/S_{\lambda}$ for a partition $\lambda$ of $n$ to the one of computing the $E$-polynomial of $C_n(X)/S_n$.

The second method uses combinatorial techniques, and is base on the description of the configuration space of points as a cartesian product with the \emph{generalized diagonals} removed. This will give us a new formula for computing the $E$-polynomial of $C_n(X)$ in a recursive way. Moreover, the description of $C_n(X)$ in terms of the generalized diagonals will allow us later to compute the desired $E$-polynomial of $C_n(X)/S_n$.

\subsection{Geometric method}
This method is based on the following result.
\begin{tma}\label{fibrationCn}
    Let $X$ be a smooth connected complex algebraic variety of positive dimension, and let $m\le n\in \mathbb{N}$. Let $p_{m}^n:X^n\to X^m$ be the projection onto the first $m$ coordinates. Then, the restriction 
    \begin{equation*}
    \pi_m^n=p_m^n|_{C_n(X)}:C_n(X)\to C_m(X)
    \end{equation*}
    is a locally trivial fibration in the analytic topology with trivial monodromy and fibre $C_{n-m}(X-\{P_1,\ldots,P_m\})$ for certain $P_1,\ldots,P_m\in X$. 
\end{tma}

To prove this theorem we will need a technical lemma.

\begin{lema}
    Let $x=(x_1,\ldots,x_m)\in C_m(X)$. There exist disjoint open neighborhoods $V_{x_i}$ of the $x_i$ such that for each $y=(y_1,\ldots,y_m)\in \prod_{i=1}^mV_{x_i}$, there exists a diffeomorphism $\Psi_y:X\to X$ such that $\Psi_y(y_i)=x_i$ and $\Psi_y$ is the identity out of an open set $U_x$ containing $\bigcup_{i=1}^mV_{x_i}$. Moreover, if we call $V_x=\prod_{i=1}^m V_{x_i}$, the map $\Phi_x:(\pi_m^n)^{-1}(V_x)\to V_x\times C_{n-m}(X-\{x_1,\ldots,x_m\}), (y_1,\ldots, y_n)\mapsto (y_1,\ldots,y_m,\Psi_y(y_{m+1}),\ldots, \Psi_y(y_n))$ is a diffeomorphism.
\end{lema}

\begin{proof}
    First of all, let us define for each $x=(x_1,\ldots,x_m)\in C_m(X)$ the set of pair $\{(U_{x_i},\varphi_{x_i})\}_{i=1,\ldots,m}$, where the $U_{x_i}$ are disjoint open neighbourhoods of the $x_i$ and $\varphi_{x_i}:\mathbb{R}^p\to U_{x_i}$ is a chart with $\varphi_{x_i}(0)=x_i$. We denote by $B_r$ the ball of radius $r$ centered in $0$. Let $\theta:\mathbb{R}^p\to [0,1]$ be a plateau function such that $\theta|_{B_1}\equiv 1$ and $\theta|_{\mathbb{R}^p-B_2}\equiv 0$. We define for each $v\in\mathbb{R}^p$ the map $t_v(u)=u+\theta(u^2)v$, which is clearly smooth. Let $M=\max\{|d_{u}\theta(v)| \ | \ u,v \in \mathbb{R}^p, \|v\|=1\}$, which exists since $d_u\theta$ is null if $u\notin B_2$. We are going to see that if $v\in B_{1/M}$, then the map $t_v$ is a diffeomorphism. Indeed, for each $P\in \mathbb{R}^p$, we can consider the restriction $t_v|_{L_P}:L_P\to L_P$, where $L_P=P+L[v]$, and for proving the injectivity of $t_v$ is enough to prove the injectivity of these restrictions. But, through the isomorphism $\mathbb{R}\to L_P, s\mapsto P+sv$, we can identify them with the one variable map $t_{v,P}(s)=s+\theta(P+sv)$. We have that $t'_{v,P}(s)=1+d_{P+sv}\theta(v)\geq 1-M\|v\|>0$, which proves that $t_{v,P}$ is injective as desired. So for proving that these maps are diffeomorphisms, it is enough to prove that they are local diffeomorphisms, but this is easy, since, in a basis $\{v_1=v,v_2,\ldots, v_p\}$, the Jacobian matrix of $t_v$ has the form
    \begin{equation*}
        \left(\begin{array}{c|c}
        1+d_u\theta(v) & *\\
        \hline
        0 & \mathrm{Id}_{p-1}
        \end{array}\right),
    \end{equation*}
    so its determinant is $1+d_u\theta(v)>0$, and the result follows from the Local Inverse Theorem.
    
    Now, for each $y_i\in \bigcup_{i=1}^mV_{x_i}$, we define the map $\psi_{y_i}=\varphi_{x_i}\circ t_{-\varphi^{-1}(y_i)}\circ\varphi_{x_i}^{-1}:U_{x_i}\to U_{x_i}$. We observe that $\psi_{y_i}$ is a diffeomorphism that sends $y_i$ to $x_i$. Let us now define, for each $y=(y_1,\ldots,y_m)\in \bigcup_{i=1}^mV_{x_i}$, the diffeomorphism $\Psi_y:X\to X$ given by
    \begin{equation*}
        \Psi_y(z)=\left\{\begin{array}{cl}
        \psi_{y_i}(z) & \text{ if }z\in U_{x_i},\\
        z & \text{ otherwise}.
        \end{array}\right.
    \end{equation*}
    Finally, the diffeomorphism between $(\pi_m^n)^{-1}(V_x)$ and $V_x\times C_{n-m}(X-\{x_1,\ldots,x_m\})$ is given by the map
    \begin{IEEEeqnarray*}{rCl}
        \Phi_x:(\pi_m^n)^{-1}(V_x)&\to& X_x\times C_{n-m}(X-\{x_1,\ldots,x_m\})\\
        (y_1,\ldots,y_n)&\mapsto& (y_1,\ldots, y_m,\Psi_{y}(y_{m+1}),\ldots, \Psi_{y}(y_n)).
    \end{IEEEeqnarray*}
\end{proof}

\begin{proof}[Proof of Theorem \ref{fibrationCn}]

    Let us see that the collection $\{V_x\}_{x\in X}$ of the open sets of the previous lemma is a covering of $C_m(X)$ by trivializing neighbourhoods. By the previous lemma, we know that each of the $(\pi_m^n)^{-1}(V_x)$ is diffeomorphic to $V_x\times C_{n-m}(X-\{x_1,\ldots,x_m\})$ through the diffeomorphism $\Phi_x$. Now, we just need to compose this diffeomorphism with a diffeotopy that sends $\{x_1,\ldots,x_m\}$ to $\{P_1,\ldots,P_m\}$ to have the diffeomorphism with $V_x\times C_{n-m}(X-\{P_1,\ldots,P_m\})$.
    This proves that $\pi_m^n$ is a fiber bundle.

    Let us now look at the monodromy. Let $\gamma:[0,1]\to C_m(X)$ be a loop with basepoint $(x_1,\ldots,x_m)$. For each $(x_1,\ldots,x_n)\in(\pi_m^n)^{-1}(x_1,\ldots,x_m)$, we must find a lifting $\widetilde{\gamma}$ of $\gamma$ with $\widetilde{\gamma}(0)=\widetilde{\gamma}(1)=(x_1,\ldots,x_n)$. The loop $\gamma$ defines loops $\gamma_1,\ldots,\gamma_m$ in $X$ with basepoints $x_1,\ldots,x_m$ respectively and such that $\gamma_i(t)\neq \gamma_j(t)$ for all $t\in [0,1]$ if $i\neq j$ for $i,j=1,\dots,m$. For finding the desired lifting, we just need to find loops $\gamma_{m+1},\ldots,\gamma_n$ in $X$ with basepoints $x_{m+1},\ldots,x_n$ and such that $\gamma_i(t)\neq \gamma_j(t)$ for all $t\in [0,1]$ if $i\neq j$ for $i,j=1,\ldots,n$. We are going to show how to construct the path $\gamma_i$ supposed that we have $\gamma_1,\ldots,\gamma_{i-1}$.

    If $x_i\notin \gamma_j([0,1])$ for any $j=1,\ldots,i-1$, then it is enough to define $\gamma_i$ as the constant path $\gamma_i(t)=x_i$. Otherwise, define $t_0$ and $t_1$ as 
    \begin{eqnarray*}
        t_0=\min\,\{t\in[0,1] \ | \ \gamma_j(t)=x_i \ \text{ for some }j=1,\ldots,i-1\} \ \\
        t_1=\max\,\{t\in[0,1] \ | \ \gamma_j(t)=x_i \ \text{ for some }j=1,\ldots,i-1\}.
    \end{eqnarray*}
    Since $X$ is a complex variety of positive dimension, its real dimension is at least $2$, so we can find a point $y\in X$ and a path $\alpha:[0,1]\to X$ with $\alpha(0)=x_1$ and $\alpha(1)=y$ and such that $\alpha([0,1])\cap\left(\bigcup_{j=1}^{i-1}\gamma_j([0,1])=\\{x_i}\right)$. We construct the path $\gamma_i$ as
    \begin{equation*}
        \gamma_i(t)=\left\{\begin{array}{cl}
        \alpha(t/t_0) & \text{ if } t\in [0,t_0],\\
        y& \text{ if } t\in [t_0,t_1],\\
        \alpha((1-t)/(1-t_1)) & \text{ if }t\in [t_1,1].
        \end{array}\right.
    \end{equation*}
    The loop $\gamma_i$ satisfies the required conditions.

    By doing this with each $i\in\{m+1,\ldots,n\}$ we get the desired lifting, so we conclude that the action of $\pi_1(C_m(X))$ in $(\pi_m^n)^{-1}(x_1,\ldots,x_m)$ is trivial for each $(x_1,\ldots,x_m)\in C_m(X)$.
\end{proof}

To address the case where $C_n(X)$ is not smooth, we will first prove a technical result.

\begin{prop}\label{stratifiedfibration}
    Let $\pi:X\to Y$ be a surjective regular map between algebraic varieties and let $H$ be a finite group such that for every $y_1,y_2\in Y$, if we call $F_1=\pi^{-1}(y_1)$ and $F_2=\pi^{-1}(y_2)$, then $e_H(F_1)=e_H(F_2)$. Suppose that $X=\bigsqcup_{i=1}^r X_i$ is stratified as a disjoint union of varieties such that for all $j=1,\ldots,r$, the strata $X_j$ is an open subvariety of $\bigsqcup_{i=j}^r X_i$, and the restriction $\pi|_{X_i}:X_i\to \pi(X_i)$ is a locally trivial fibration in the analytic topology with trivial monodromy. Then
    \begin{equation*}
        e_H(X)=e_H(Y)\otimes e_H(F),
    \end{equation*}
    where $F$ is the inverse image through $\pi$ of a point of $Y$.
\end{prop}

\begin{proof}
We can suppose that for all $X_i$ and $X_j$, we have that $\pi(X_i)\cap\pi(X_j)=\varnothing$ or $\pi(X_i)=\pi(X_j)$. Indeed, otherwise we can suppose that $\pi(X_i)\cap\pi(X_j)\subsetneq \pi(X_i)$. Stratifying $X_i$ as $X_i=X_i^1\sqcup X_i^2$ with $X_i^1=\pi^{-1}\left(\pi(X_i)\cap \pi(X_j)\right)\cap X_i$ and $X_i^2=X_i-X_i^1$, we get that $\pi(X_i^1)\cap\pi(X_j)=\pi(X_i^1)$ and $\pi(X_i^2)\cap\pi(X_j)=\varnothing$. If $\pi(X_i^1)\subsetneq \pi(X_j)$, we repeat this process with $X_j$, obtaining two varieties $X_j^1,X_j^2$ such that $X_j=X_j^1\sqcup X_j^2$ and $\pi(X_i^1)=\pi(X_j^1)$ and $\pi(X_i^1)\cap\pi(X_j^2)=\varnothing$. In any case, the restrictions of $\pi$ to the new varieties are fiber bundles, so we are still in the conditions of the proposition.

So let us suppose that $\pi(X_i)\cap\pi(X_j)=\varnothing$ or $\pi(X_i)=\pi(X_j)$. Let $\{Y_1,\ldots,Y_s\}=\{\pi(X_i)\}_{i=1,\ldots,r}$ and let $F'_i$ be the fiber of $\pi|_{X_i}:X_i\to\pi(X_i)$. By the additivity of $E$-polynomials and Proposition \ref{propeqvar}, we get that
\begin{equation}\label{eHsum}
    e_H(X)=\sum_{i=1}^re_H(X_i)=\sum_{i=1}^re_H(\pi(X_i))\otimes e_H(F'_i)=\sum_{j=1}^se_H(Y_j)\otimes\sum_{\pi(X_i)=Y_j}e_H(F'_i).
\end{equation}
But for each $Y_j$, we have that for all $y\in Y_j$, 
\begin{equation}\label{eHfiber}
e_H\left(\pi^{-1}(y)\right)=\sum_{\pi(X_i)=Y_j}e_H\left(\pi^{-1}(y)\cap X_i\right)=\sum_{\pi(X_i)=Y_j}e_H(F'_i),
\end{equation}
and substituting equation (\ref{eHfiber}) in equation (\ref{eHsum}), we obtain
\begin{equation*}
    e_H(X)=\sum_{j=1}^se_H(Y_j)\otimes e_H(F)=e_H(Y)\otimes e_H(F).
\end{equation*}
\end{proof}

\begin{coro}\label{CntoCm}
    Let $X$ be a complex algebraic variety, and let $m\leq n$ be two natural numbers. Let $H$ be a finite group that acts on $C_n(X)$ and consider the induced action on $C_m(X)$ by $\pi_m^n$. Then 
    \begin{equation*}
    e_H\left(C_n(X)\right)=e_H\left(C_m(X)\right)\otimes e_H\left(C_{n-m}(X-\{P_1,\ldots,P_m\})\right).
    \end{equation*}
\end{coro}

\begin{proof}
    First of all, we know that we can stratify our variety $X$ as $X=X_1\sqcup\ldots\sqcup X_k$ where all the $X_i$ are smooth. Indeed, we just need to stratify $X$ as $X=X^S\sqcup X^{NS}$, where $X^S$ is the locus of singular points of $X$, which is a variety of dimension strictly less than $\dim_{\mathbb{C}}(X)$, and $X^{NS}$ is the smooth open subvariety of smooth points of $X$. We now call $X_1=X^{NS}$ and repeat the process with $X^S$ until we get the desired decomposition. 

    We are now going to prove the result by induction on the number $k$ of strata.
    
    \begin{itemize}
    \item If $k=1$, then $X$ is smooth, and the result follows from Theorem \ref{fibrationCn}. 
    
    \item Suppose now that the result is true when our variety is stratified in $k-1$ strata and let us prove it for $k$ strata. Let $X'=X_1\sqcup\ldots\sqcup X_{k-1}$. We have that $C_n(X)=\sqcup_{p+q=n}\binom{n}{p}C_p(X')\times C_q(X_k)$, where the coefficient $\binom{n}{p}$ before $C_p(X')\times C_q(X_k)$ means that this strata appears $\binom{n}{p}$ times in the stratification. 
    
    It is obvious that the restriction $\pi_m^n|_{C_p(X')\times C_q(X_k)}$ is $\pi_{m_1}^p\times \pi_{m_2}^q$ for some $m_1,m_2\in \mathbb{N}$, where $\pi_{m_1}^p:C_p(X')\to C_{m_1}(X')$ and $\pi_{m_2}^p:C_p(X_k)\to C_{m_2}(X_k)$. By induction hypothesis, we have that 
    \begin{equation*}
    e_H(C_p(X'))=e_H(C_{m_1}(X'))\otimes e_H(C_{p-m_1}(X'-\{Q_1,\ldots,Q_{m_1}\}))
    \end{equation*}
    and 
    \begin{equation*}
        e_H(C_p(X_k))=e_H(C_{m_2}(X_k))\otimes e_H(C_{q-m_2}(X_k-\{R_1,\ldots,R_{m_2}\})),
    \end{equation*}
    so we have 
    \begin{align*}
        e_H(C_p(X')\times C_q(X_k))&=e_H(C_{m_1}(X')\times C_{m_2}(X_k))\\
        &\otimes e_H\left(C_{p-m_1}(X'-\{Q_1,\ldots,Q_{m_1}\})\right.\\
        &\left.\qquad\times C_{q-m_2}(X_k-\{R_1,\ldots,R_{m_2}\})\right).
    \end{align*}
    Therefore, by Proposition \ref{stratifiedfibration}, 
    \begin{equation*}
    e_H\left(C_n(X)\right)=e_H\left(C_m(X)\right)\otimes e_H\left(C_{n-m}(X-\{P_1,\ldots,P_m\})\right).
    \end{equation*}
    \end{itemize}
\end{proof}

With all this, we are in the conditions to prove the announced formula for the $E$-polynomial, which coincides with the one given in \cite{bibby2022combinatorics}.

\begin{coro}\label{geometh}
    Let $X$ be an algebraic variety. Then, the $E$-polynomial of $C_n(X)$ is
    \begin{equation*}
        e(C_n(X))=\prod_{i=0}^{n-1}(e(X)-i).
    \end{equation*}
\end{coro}

\begin{proof}
    Just apply Corollary \ref{CntoCm} to the map $\pi_1^n:C_n(X)\to C_1(X)$ and proceed recursively on the fiber.
\end{proof}

\begin{ejem}\label{geoejem} We are going to use this method to compute the $E$-polynomial of the configuration space of $\PGL_2=\PGL_2(\mathbb{C})$. The $E$-polynomial of $\PGL_2$ is 
\begin{equation}
    e(\PGL_2)=q^3-q.
\end{equation}
So by Corollary \ref{geometh}, we get the following formulas:
\begin{itemize}
    \item For $n=1$ we just have $C_1(\PGL_2)=\PGL_2$, so $e(C_1(\PGL_2))=q^3-q$.
    \item For $n=2$, Corollary \ref{geometh} gives us 
    \begin{equation*}
        e(C_2(\PGL_2))=e(\PGL_2)(e(\PGL_2)-1)=(q^3-q)(q^3-q-1)=q^6 - 2q^4 - q^3 + q^2 + q.
    \end{equation*}
    \item For $n=3$, we get
        \begin{align*}
        e(C_3(\PGL_2))&=e(\PGL_2)(e(\PGL_2)-1)(e(\PGL_2)-2)=(q^3-q)(q^3-q-1)(q^3-q-2)\\&=q^9 - 3q^7 - 3q^6 + 3q^5 + 6q^4 + q^3 - 3q^2 - 2q.
    \end{align*}
\end{itemize}

With the aid of a computer algebra system, we can continue for an arbitrary number of points. To illustrate the kind of results we can obtain, for $n \leq 7$ we get:
%As the order of the configuration space gets incremented, the computations become more and more difficult but, with the help of an algebra assistant program, we can compute the $E$-polynomial for high order configuration spaces. We give some examples:
\begin{itemize}
    \item $e(C_4(\PGL_2))=q^{12} - 4q^{10} - 6q^9 + 6q^8 + 18q^7 + 7q^6 - 18q^5 - 21q^4 + 11q^2 + 6q$.\\
    \item $e(C_5(\PGL_2))=q^{15} - 5q^{13} - 10q^{12} + 10q^{11} + 40q^{10} + 25q^9 - 60q^8 - 100q^7 - 10q^6 + 104q^5 + 90q^4 - 11q^3 - 50q^2 - 24q$.\\
    \item $e(C_6(\PGL_2))=q^{18} - 6q^{16} - 15q^{15} + 15q^{14} + 75q^{13} + 65q^{12} - 150q^{11} - 325q^{10} - 75q^9 + 504q^8 + 600q^7 - 65q^6 - 660q^5 - 463q^4 + 105q^3 + 274q^2 + 120q$.\\
    \item $e(C_7(\PGL_2))=q^{21} - 7q^{19} - 21q^{18} + 21q^{17} + 126q^{16} + 140q^{15} - 315q^{14} - 840q^{13} - 315q^{12} + 1729q^{11} + 2625q^{10} - 119q^9 - 4284q^8 - 3998q^7 + 1155q^6 + 4697q^5 + 2793q^4 - 904q^3 - 1764q^2 - 720q$.
\end{itemize}
%We stop here as we think these examples are enough to illustrate the kind of results we can obtain, but we could have continued for much higher orders.
\end{ejem}

\subsection{Combinatorial method} In this section, we are going to perform the same computation, but in a more combinatorial way. First of all, we need to introduce some notation.

We write $\underline{n}=\{1,2,\ldots,n\}$. Given a partition $\lambda=(\lambda_1,\ldots,\lambda_l)$ of $n$ with $\lambda_1\geq\lambda_2\geq\ldots\geq\lambda_l$, we say that a partition $I$ of the set $\underline{n}$ is associated to $\lambda$ if it is of the form $I=\{\{i_1^1,\ldots,i_{\lambda_1}^1\},\ldots,\{i_1^l,\dots,i_{\lambda_l}^l\}\}$. We denote $\mathcal{I}_{\lambda}(n)$ the set of all partitions of $\underline{n}$ associated to the partition $\lambda$ of $n$.

\begin{defi}
    Let $X$ be an algebraic variety and let $I$ be a partition of $\underline{n}$. We define the \emph{generalized diagonal associated to $I$} as
    \begin{equation*}
        C_n^I(X):=\{(x_1,\dots,x_n) \in X^n \ | \ x_{i_r^t}=x_{i_s^u}\Leftrightarrow t=u\}.
    \end{equation*}
\end{defi}

\begin{remark}\label{CnIisomorphic}
    Given a partition $\lambda$ of $n$ and two partitions $I=\{\{i_1^1,\ldots,i_{\lambda_1}^1\},\ldots,\{i_1^l,\dots,i_{\lambda_l}^l\}\}$ and $J=\{\{j_1^1,\ldots,j_{\lambda_1}^1\},\ldots,\{j_1^l,\dots,j_{\lambda_l}^l\}\}$ associated to $\lambda$, the map
    \begin{equation}\label{sigmaIJ}
        \sigma_{I,J}:C_n^I(X)\to C_n^J(X)
    \end{equation}
    that maps the coordinate $x_{i_r^s}$ to the coordinate $x_{j_r^s}$ is an isomorphism, so all the generalize diagonals associated to a partition of $\mathcal{I}_{\lambda}(n)$ are isomorphic. Moreover, it is obvious that they are isomorphic to $C_l(X)$, where $l$ is the length of the partition $\lambda$, through the map
    \begin{equation}\label{sigmaI}
        \sigma_I:C_n^I(X)\to C_l(X)
    \end{equation}
    given by $\sigma_I(x_1,\ldots,x_n)=(x_{i_1^1},\ldots, x_{i_1^l})$.
\end{remark}

Let $\mathcal{P}(n)$ be the set of all partitions of $n$, and denote $\mathcal{P}(n)^0=\mathcal{P}(n)-\{(1,1,\ldots,1)\}$. We have that
\begin{equation}\label{eCn1}
    C_n(X)=X^n-\bigsqcup_{\lambda\in\mathcal{P}(n)^0}\bigsqcup_{I\in \mathcal{I}_{\lambda}(n)}C_n^I(X),
\end{equation}
so 
\begin{equation*}
    e\left(C_n(X)\right)=e(X^n)-\sum_{\lambda\in\mathcal{P}(n)^0}\sum_{I\in \mathcal{I}_{\lambda}(n)}e\left(C_n^I(X)\right).
\end{equation*}

If we denote by $(\lambda_1^{k1},\ldots,\lambda_p^{k_p})$ the partition where $\lambda_i\neq \lambda_j$ if $i\neq j$ and each $\lambda_i$ appears $k_i$ times, then there are $\frac{n!}{(\lambda_1!)^{k_1}k_1!\cdots (\lambda_p!)^{k_p}k_p!}$ partitions of $\underline{n}$ associated to a partition $(\lambda_1^{k_1},\ldots,\lambda_p^{k_p})$ of $n$. Together with Remark \ref{CnIisomorphic} and equation (\ref{eCn1}), this gives us that
\begin{equation}\label{eCn2}
    e\left(C_n(X)\right)=e(X^n)-\sum_{(\lambda_1^{k_1},\ldots,\lambda_p^{k_p})\in\mathcal{P}(n)^0}\frac{n!}{(\lambda_1!)^{k_1}k_1!\cdots (\lambda_p!)^{k_p}k_p!}e\left(C_l(X)\right),
\end{equation}
where $l=\sum_{i=1}^pk_i$.

Equation (\ref{eCn2}) allows us to compute the $E$-polynomial of $C_n(X)$ in a recursive way. 

\begin{ejem}\label{combejem} Let us now use this formula for computing the $E$-polynomial of $C_n(\PGL_2)$.
\begin{itemize}
    \item For $n=2$, equation (\ref{eCn2}) gives us
    \begin{equation*}
        e(C_2(\PGL_2))=e(\PGL_2)^2-e(\PGL_2)=(q^3-q)^2-(q^3-q)=q^6 - 2q^4 - q^3 + q^2 + q.
    \end{equation*}
    %which coincides with the result of Example \ref{geoejem}.
    \item For $n=3$, the $E$-polynomial is given by
    \begin{align*}
        e(C_3(\PGL_2))&=e(\PGL_2)^3-3e(C_2(\PGL_2))-e(\PGL_2)\\&=(q^3-q)^3-3(q^6 - 2q^4 - q^3 + q^2 + q)-(q^3-q)\\
        &=q^9 - 3q^7 - 3q^6 + 3q^5 + 6q^4 + q^3 - 3q^2 - 2q.
    \end{align*}
    %which again coincides with the $E$-polynomial obtained in Example \ref{geoejem}.
\end{itemize}
Notice that the three of them agree with the calculations of Example \ref{geoejem}. Furthermore, again with the aid of a computer algebra system, we can perform the calculation for as many points as desired. For $n \leq 7$, we get the following results, in agreement with Example \ref{geoejem}.
%As it happened with the geometric method, the computations become much harder as we increase the order of the space. But again, with a computer program, we can compute $E$-polynomials for very high order. Again, we will show some of the $E$-polynomials we can compute with this method, realising that they are equal to the ones obtained with the geometric method: 

\begin{itemize}
    \item $e(C_4(\PGL_2))=q^{12} - 4q^{10} - 6q^9 + 6q^8 + 18q^7 + 7q^6 - 18q^5 - 21q^4 + 11q^2 + 6q$.\\
    \item $e(C_5(\PGL_2))=q^{15} - 5q^{13} - 10q^{12} + 10q^{11} + 40q^{10} + 25q^9 - 60q^8 - 100q^7 - 10q^6 + 104q^5 + 90q^4 - 11q^3 - 50q^2 - 24q$.\\
    \item $e(C_6(\PGL_2))=q^{18} - 6q^{16} - 15q^{15} + 15q^{14} + 75q^{13} + 65q^{12} - 150q^{11} - 325q^{10} - 75q^9 + 504q^8 + 600q^7 - 65q^6 - 660q^5 - 463q^4 + 105q^3 + 274q^2 + 120q$.\\
    \item $e(C_7(\PGL_2))=q^{21} - 7q^{19} - 21q^{18} + 21q^{17} + 126q^{16} + 140q^{15} - 315q^{14} - 840q^{13} - 315q^{12} + 1729q^{11} + 2625q^{10} - 119q^9 - 4284q^8 - 3998q^7 + 1155q^6 + 4697q^5 + 2793q^4 - 904q^3 - 1764q^2 - 720q$.
\end{itemize}
\end{ejem}

\subsection{$E$-polynomial of $C_n(X)/S_n$} As we announced at the beginning of the section, we are now going to modify the combinatorial method introduced before to study the quotient by the action of $S_n$. For that purpose, we firstly observe that given a partition $\lambda=(\lambda_1,\ldots,\lambda_l)\in \mathcal{P}(n)$, all the varieties $C_n^I(X)$ for $I\in \mathcal{I}_{\lambda}(n)$ are identified in the quotient $C_n(X)/S_n$ as the variety $C_l(C)/S_l$. More in detail, consider the unique morphism
\begin{equation*}
    \sigma_{\lambda}:\bigsqcup_{I\in\mathcal{I}_{\lambda}(n)}C_n^I(X)\to C_l(X)
\end{equation*}
such that $\sigma_{\lambda}|_{C_n^I(X)}=\sigma_I$ for all $I\in\mathcal{I}_{\lambda}(n)$. If we call $p_l:C_l(X)\to C_l(X)/S_l$ the quotient map, we have that the composition $p_l\circ\sigma_{\lambda}$ is $S_n$-equivariant, so it descends to a regular map 
\begin{equation*}
    \overline{\sigma}_{\lambda}:\left(\bigsqcup_{I\in\mathcal{I}_{\lambda}(n)}C_n^I(X)\right)/S_n\to C_l(X)/S_l
\end{equation*}
such that the following diagram commutes:
\begin{equation*}
    \xymatrix{
    \bigsqcup_{I\in\mathcal{I}_{\lambda}(n)}C_n^I(X)\ar[r]^-{\sigma_{\lambda}}\ar[d] & C_l(X)\ar[d]^-{p_l}\\
    \left(\bigsqcup_{I\in\mathcal{I}_{\lambda}(n)}C_n^I(X)\right)/S_n\ar[r]^-{\overline{\sigma}_{\lambda}} & C_l(X)/S_l.
    }
\end{equation*}
It can be easily seen that $\overline{\sigma}_{\lambda}$ is bijective, so by \cite{Gonz_lez_Prieto_2023}, we have
\begin{equation*}
    e\left(\left(\bigsqcup_{I\in\mathcal{I}_{\lambda}(n)}C_n^I(X)\right)/S_n\right)=e\left(C_l(X)/S_l\right).
\end{equation*}
We conclude the following
\begin{tma}\label{CnquotSn}
    Let $X$ be an algebraic variety and consider the action of $S_n$ on $C_n(X)$ by permutation of the coordinates. The $E$-polynomial of the quotient $C_n(X)/S_n$ is 
    \begin{equation*}
        e\left(C_n(X)/S_n\right)=e\left(\mathrm{Sym}^n(X)\right)-\sum_{l=1}^{n-1}p(n,l)e\left(C_l(X)/S_l\right),
    \end{equation*}
    where $p(n,l)$ is the number of partitions of $n$ of length $l$ and $\mathrm{Sym}^n(X) = X^n/S_n$ is the $n$-th symmetric product of $X$.
\end{tma}

\begin{remark}
\begin{enumerate}
   \item It is well known that the numbers $p(n,l)$ can be computed recursively using the facts that $p(n,n)=1$ and $p(n,l)=p(n-1,l-1)+p(n-l,l)$.
   \item The $E$-polynomial of $\mathrm{Sym}^n(X)$ can be computed with the so-called plethystic exponential, since
   \begin{equation*}
       \sum_{n\geq0}e(\mathrm{Sym}^n(X))t^n=\mathrm{PExp}(e(X)t).
   \end{equation*}
   For more information, see \cite[Proposition 4.6]{florentino2019generating}.
   \end{enumerate}
\end{remark}

\begin{ejem}\label{CnquotSnejem}
    We can now apply Theorem \ref{CnquotSn} to the case of $\PGL_2$, obtaining the following formulas:
    \begin{itemize}
        \item $e(C_2(\PGL_2))/S_2=q^{6} - q^{4} - q^{3} + q$.\\
        \item $e(C_3(\PGL_2))/S_3=q^{9} - q^{7} - q^{6} + q^{4}$.\\
        \item $e(C_4(\PGL_2))/S_4=q^{12} - q^{10} - q^{9} + q^{7} - q^{6} + q^{4} + q^{3} - q$.\\
        \item $e(C_5(\PGL_2))/S_5=q^{15} - q^{13} - q^{12} + q^{10} - q^{9} + q^{7} + q^{6} - q^{4}$.\\
        \item $e(C_6(\PGL_2))/S_6=q^{18} - q^{16} - q^{15} + q^{13} - q^{12} + q^{10} + q^{6} - q^{4}$.\\
        \item $e(C_7(\PGL_2))/S_7=q^{21} - q^{19} - q^{18} + q^{16} - q^{15} + q^{13} + q^{9} - q^{7} + q^{6} - q^{4} - q^{3} + q$.\\
    \end{itemize}
\end{ejem}

\begin{remark}
    In this example, we can clearly observe a stability of the coefficients of the $E$-polynomial. This suggests that the Hodge structure of the quotients of configuration spaces of points satisfy some kind of  stability. In this sense, in \cite{church2012homological} it was proven that these quotients satisfied a homological stability by proving what they called a \emph{representation stability} for the configuration spaces before quotienting. We believe that the techniques used in the cited paper can be adapted to prove a similiar stability of the Hodge structure.
\end{remark}
    
\section{Generalized plus-minus formula}\label{pmformula}
In \cite{logares2013hodge}, it was proved that given a locally trivial fibration in the analytic topology $F\to E\to B$ with trivial monodromy, the $E$-polynomial of the quotient $E/\mathbb{Z}_2$ was given by the formula
\begin{equation*}
    e^+(E)=e(B)^+e(F)^++e(B)^-e(F)^-,
\end{equation*}
where $e(X)^+=e(X/\mathbb{Z}_2)$ and $e(X)^-=e(X)-e(X)^+$.
In this section, we will give a general method for computing the $E$-polynomial of $E/S_n$ for an arbitrary $n$. Following Proposition \ref{propeqvar}, we just need to find a general method for computing the equivariant $E$-polynomial $e_{S_n}(X)$ of an algebraic variety $X$ since, known $e_{S_n}(B)$ and $e_{S_n}(F)$, their product $e_{S_n}(B)\otimes e_{S_n}(F)$ can be easily computed using the character table of $S_n$.

So let $X$ be an algebraic variety. We will prove that its equivariant $E$-polynomial is uniquely determined by the $E$-polynomials of the quotients $X/S_{\mu}$, where for a partition $\mu=(\mu_1,\ldots,\mu_l)$ of $n$, $S_{\mu}$ denotes the subgrup $S_{\mu_1}\times\cdots\times S_{\mu_l}<S_n$. Identifying each irreducible representation of $S_n$ with its character, we can write
\begin{equation*}
    e_{S_n}(X)=\sum_{\lambda\in\mathcal{P}art(n)}a_{\lambda}\chi_{\lambda}.
\end{equation*}
If we now call $\chi_{\lambda}^{\mu}$ to the dimension of the subspace $V_{\lambda}^{\mu}$ of $V_{\lambda}$ invariant under the action of $S_{\mu}$, then part (a) of Proposition \ref{propeqvar} gives us the formula
\begin{equation*}
    e(X/S_{\mu})=\sum_{\lambda\in\mathcal{P}(n)}a_{\lambda}\chi_{\lambda}^{\mu}.
\end{equation*}

Therefore, we have a system of $|\mathcal{P}(n)|$ linear equations
\begin{equation}\label{systemofquotients}
    \left\{e(X/S_{\mu})=\sum_{\lambda\in\mathcal{P}(n)}a_{\lambda}\chi_{\lambda}^{\mu}\right\}_{\mu\in\mathcal{P}(n)},
\end{equation}
with the $a_{\lambda}$ as unknowns. In Theorem \ref{matrixofcaracters}, we shall prove that the matrix of coefficients of this system is invertible and thus, to compute $e_{S_n}(X)$ we just need to solve the system (\ref{systemofquotients}).

To prove such a result, we will need some previous results. Before stating them, let us fix some notation. We say that a partition $\mu'=(\mu'_1,\ldots,\mu'_s)$ of $n$ refines another partition $\mu=(\mu_1\ldots,\mu_l)$ if there exist $l$ indexes $1\leq s_1<s_2<\ldots<s_l= s$ such that for all $i=1,\ldots,l$, $(\mu'_{s_{i-1}+1},\ldots,\mu'_{s_{i}})$ is a partition of $\mu_i$. We say that a permutation $\sigma\in S_n$ has cycle structure associated to a partition $\mu=(\mu_1\ldots,\mu_l)$ of $n$ if its decomposition into disjoint cycles is of the form
\begin{equation*}
    (n_1\ \ldots\ n_{\mu_1})(n_{\mu_1+1}\ \ldots \ n_{\mu_2})\cdots(n_{\sum_{=1}^{l-1}\mu_i+1}\ \ldots \ n_{\mu_l}).
\end{equation*}
Now, let $\mathrm{Ind}_{S_{\mu}}(1)$ be the representation of $S_n$ induced by the trivial representation $1$ of $S_{\mu}$. Let us call $\chi_{S_{\mu}}=\chi(\mathrm{Ind}_{S_{\mu}}(1))$ to the character of this representation. We have

\begin{prop}\label{charvalue}
    The character $\chi_{S_{\mu}}$ is non-negative and it satisfies that $\chi_{S_{\mu}}(g)>0$ if and only if the cycle structure of $g$ is given by a partition $\mu'$ that refines $\mu$. 
\end{prop}

\begin{proof}
    We are firstly going to recall how to construct $\mathrm{Ind}_{S_{\mu}}(1)$. The trivial representation $1$ is given by the trivial action of $S_{\mu}$ over the space $\mathbb{R}$. Let $S_n/S_{\mu}=\{\sigma_1S_{\mu},\sigma_2S_{\mu},\ldots, \sigma_rS_{\mu}\}$. Then, the induced representation $\mathrm{Ind}_{S_{\mu}}(1)$ is given by the vector space
    \begin{equation*}
        W=\bigoplus_{i=1}^r\sigma_i\mathbb{R}
    \end{equation*}
    with $S_n$ acting on an element $\sigma_iv\in\sigma_i\mathbb{R}$ by 
    \begin{equation*}
    \sigma\cdot \sigma_iv=\sigma_j(\tau\cdot v),
    \end{equation*}
    where $\sigma_j\in S_n$ and $\tau\in S_{\mu}$ are the only elements such that $\sigma\sigma_i=\sigma_j\tau$. But since the action of $S_{\mu}$ is trivial over $\mathbb{R}$, we have $\sigma\cdot \sigma_iv=\sigma_j$. 
    
    Therefore, if we consider the basis $B=\{\sigma_11,\ldots,\sigma_r1\}$ of $W$, then an element $\sigma\in S_n$ acts over $B$ by permuting its elements. So the matrix of $\sigma$ with respect to this matrix is a permutation matrix, whose trace is positive. This proves the first statement. For the second part, we observe that the trace of this matrix is positive if and only if $\sigma$ fixes some of the vectors $\sigma_i1$. This means that $\sigma\sigma_i=\sigma_i\tau$ for some $\tau\in S_{\mu}$, or, equivalently, $\sigma_i^{-1}\sigma\sigma_i=\tau$.

    So an element $\sigma\in S_n$ fixes an element of the basis $B$ if and only if it is conjugated to an element of $S_{\mu}$ (it can be easily proven that if $g^{-1}\sigma g=\tau$ for some $g\in S_n$ and $\tau\in S_{\mu}$ then, if $g\in \sigma_iS_{\mu}$, we have that $\sigma_i\sigma\sigma_i^{-1}=\tau'$ for another $\tau'\in S_{\mu}$). It is well known that two permutations of $S_n$ are conjugated if and only if they have the same cycle structure. But the elements of $S_{\mu}$ have cycles structures associated to partitions $\mu'$ that refine $\mu$, and for each $\mu'$ refining $\mu$, there is a permutation $\tau\in S_{\mu}$ whose cycles structure is $\mu'$. So $\sigma$ is conjugated to an element in $S_{\mu}$ if and only if its cycle structure satisfies this condition.
\end{proof}

\begin{prop}\label{charactersli}
    The characters $\chi_{s_{\mu}}$ are linearly independent.
\end{prop}

\begin{proof}
    The partitions of $n$ form a partially ordered set by the relation of refinement\\
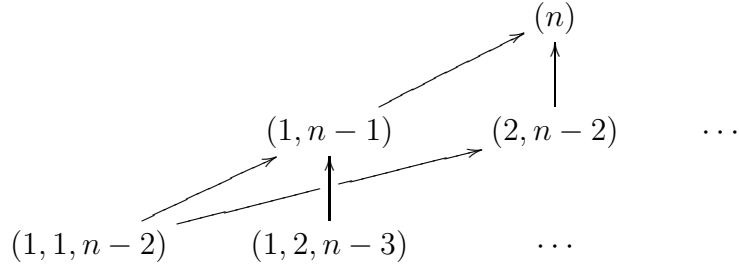
    \begin{equation*}
        \xymatrix{
        &&(n)&\\
        &(1,n-1)\ar[ur]&(2,n-2)\ar[u]&\cdots\\
        (1,1,n-2)\ar[ur]\ar[urr]|!{[r];[ur]}\hole &(1,2,n-3)\ar[u]&\cdots&
        }
    \end{equation*}
    Suppose now that we have a linear combination 
    \begin{equation*}
        \sum_{\mu}\alpha_{\mu}\chi_{\mu}=0.
    \end{equation*}
    Take the permutation $\sigma=(1\ 2 \ \ldots \ n)$. Then, we have
     \begin{equation*}
        \alpha_{(n)}\chi_{S_{(n)}}(\sigma)+\sum_{\mu\text{ refines }(n)}\alpha_{\mu}\chi_{\mu}(\sigma)=0. 
    \end{equation*}
    By Proposition \ref{charvalue}, we know that $\chi_{S_{(n)}(\sigma)}\neq 0$ and $\chi_{\mu}(\sigma)=0$ for all $\mu\neq (n)$, from which we conclude that $\alpha_{(n)}=0$.

    Proceeding recursively and taking for each partition $\mu=(\mu_1,\ldots,\mu_k)$ the permutation 
    \begin{equation*}
    (1 \ \ldots \ \mu_1)(\mu_1+1 \ \ldots \ \mu_1+\mu_2)\ldots (\sum_{i=1}^{k-1}\mu_i+1 \ \ldots \ n),
    \end{equation*}
    we get that $a_{\mu}=0$ for all $\mu\in\mathcal{P}(n)$.
\end{proof}

We are now ready to prove the key result of the method. 

\begin{tma}\label{matrixofcaracters}
    The matrix $A=(\chi_{\lambda}^{\mu})_{\lambda,\mu}$ is invertible.
\end{tma}

\begin{proof}%[Proof of Theorem \ref{matrixofcaracters}]
    By the well-known Int-Res duality, we have that
    \begin{equation*}
    \chi_{\lambda}^{S_{\mu}}=\langle 1_{S_{\mu}},\mathrm{Res}_{S_{\mu}}(\chi_{\lambda})\rangle_{S_{\mu}}=\langle \chi_{S_{\mu}},\chi_{\lambda}\rangle_{S_n}.
    \end{equation*}
    But the latter are just the coordinates of $\chi_{S_{\mu}}$ in the basis $\chi_{\lambda}$, and, since these characters are linearly independent, the matrix $A=(\chi_{\lambda}^{\mu})_{\lambda,\mu}$ is invertible.
\end{proof}

\begin{remark}\label{murnagham}
    The matrix $A$ can be explicitly computed using the Murnagham-Nakayama rule, which describes the value of $\chi_{\lambda}(\sigma)\in\mathbb{C}$ in terms of the Young diagram, since 
    \begin{equation*}
        \chi_{\lambda}^{\mu}=\langle1_{\mu},\chi_{\lambda}\rangle=\frac{1}{|S_{\mu}|}\sum_{\sigma\in S_{\mu}}\chi_{\lambda}(\sigma).
    \end{equation*}
\end{remark}

The results of this section allow us to compute the $S_n$-equivariant $E$-polynomial of a variety as long as we know the $E$-polynomials of all the quotients $X/S_{\mu}$. In the following section, we will compute these $E$-polynomials for the case when our variety is a cartesian product of varieties or a configuration space.

\begin{ejem}
    Let us use our method to find a formula for the $S_2$-equivariant $E$-polynomial of a variety. Let $T$ be the trivial representation of $S_2$ and $N$ be the sign representation. Then the equivariant $E$-polynomial is given by $e_{S_2}(X)=aT+bN$, and we only have to compute $a$ and $b$. Let us first find our matrix. For this, we need to compute the dimensions $\chi_{(2)}^{(2)},\chi_{(1,1)}^{(2)},\chi_{(2)}^{(1,1)}$ and $\chi_{(1,1)}^{(1,1)}$. By Remark \ref{murnagham}, we have
    \begin{itemize}
        \item $\chi_{(2)}^{(2)}=\frac{1}{|S_2|}\left(\chi_{(2)}\left((1)(2)\right)+\chi_{(2)}\left((1 \ 2)\right)\right)=\frac{1}{2}(1+1)=1$,\\
        \item $\chi_{(1,1)}^{(2)}=\frac{1}{|S_2|}\left(\chi_{(1,1)}\left((1)(2)\right)+\chi_{(1,1)}\left((1 \ 2)\right)\right)=\frac{1}{2}(1-1)=0$,\\
        \item $\chi_{(2)}^{(1,1)}=\frac{1}{|S_1\times S_1|}\chi_{(2)}\left((1)(2)\right)=1$,\\
        \item $\chi_{(1,1)}^{(1,1)}=\frac{1}{|S_1\times S_1|}\chi_{(1,1)}\left((1)(2)\right)=1$.
    \end{itemize}
    So $a$ and $b$ are given by the linear system
    \begin{equation*}
        \left(\begin{array}{cc}
        1 & 0\\
        1 & 1
        \end{array}\right)\left(\begin{array}{c}
        a\\
        b
        \end{array}\right)=\left(\begin{array}{c}
        e(X/S_2)\\
        e(X)
        \end{array}\right).
    \end{equation*}
    Therefore, our method yields the following formula for the equivariant $E$-polynomial,
    \begin{equation}\label{pmformula2}
        e_{S_2}(X)=e(X/S_2)T+\left(e(X)-e(X/S_2)\right)N,
    \end{equation}
    which is the known formula for the $S_2$-equivariant $E$-polynomial.
\end{ejem}

\begin{ejem}
    With the help of a computer program, we can obtain formulas for symmetric groups of higher order. For example, for $S_3$, apart from the standard and the sign representations, we also have the standard representation $D$. Therefoere, the equivariant $E$-polynomial is $e_{S_3}(X)=aT+bD+cN$. In this case, the system of equations is
    \begin{equation*}
        \left(\begin{array}{ccc}
        1 & 0 & 0\\
        1 & 1 & 0\\
        1 & 2 & 1
        \end{array}\right)\left(\begin{array}{c}
        a\\
        b\\
        c
        \end{array}\right)=\left(\begin{array}{c}
        e(X/S_3)\\
        e(X/S_2\times S_1)\\
        e(X)
        \end{array}\right),
    \end{equation*}
    and the equivariant $E$-polynomial is
    \begin{equation*}
        e_{S_3}(X)=e(X/S_3)T+\left(e(X/S_2\times S_1)-e(X/S_3)\right)D+\left(e(X/S_3)-2e(X/S_2\times S_1)+e(X)\right)N.
    \end{equation*}
    This formula coincides with the one given in \cite{gonzalez2023representation}.
    
    For $S_4$ we now have $5$ representations: the trivial $V_{(4)}$, the sign $V_{(1,1,1,1)}$, the standard $V_{(3,1)}$ and two others, $V_{(2,1,1)}$ and $V_{(2,2)}$. In this case, our method gives us the formula
    \begin{align*}
        e_{S_4}(X)&=e(X/S_4)V_{(4)}+\left(e(X/S_3\times S_1)-e(X/S_4)\right)V_{(3,1)}\\
        &+\left(e(X/S_2\times S_2)-e(X/S_3\times S_1)\right)V_{(2,2)}\\
        &+\left(e(X/S_4)-e(X/S_3\times S_1)-e(X/S_2\times S_2)+e(X/S_2\times S_1 \times S_1)\right)V_{(2,1,1)}\\
        &+\left(-e(X/S_4)+2e(X/S_3\times S_1)+e(X/S_2\times S_2)\right. \\
        & \left.\qquad -3e(X/S_2\times S_1 \times S_1)+e(X)\right)V_{(1,1,1,1)}.
    \end{align*}
\end{ejem}

\section{$E$-polynomial of $C_n(X)/S_{\lambda}$ and $X^n/S_{\lambda}$}\label{CnSlambda}
In this section we are going to compute the $E$-polynomial of the quotients $X^n/S_{\lambda}$ and $C_n(X)/S_{\lambda}$. To do this we are going to use Theorem \ref{fibrationCn} and the following result.

\begin{lema}\label{Hactstrivial}
    Let $E\to B$ be a locally trivial fiber bundle in the analytic topology with fiber $F$ and trivial monodromy. Let $H_1$ and $H_2$ be two finite groups such that $H=H_1\times H_2$ acts over $E$. Suppose that the induced action of $H$ on $B$ and $F$ satisfies that $H_2$ acts trivially over $B$ and $H_1$ acts trivially on $F$. Then 
    \begin{equation*}
        e(E/H)=e(B/H_1)e(F/H_2).
    \end{equation*}
\end{lema}

\begin{proof}
    By Remark \ref{coeftrivial}, we know that the $E$-polynomial of $E/H$ is just the coefficient of the trivial representation of $H$ in $e_H(E)$, so we just need to prove that this coefficient is $e(B/H_1)e(F/H_2)$. By Proposition \ref{propeqvar}, we have $e_H(E)=e_H(B)\otimes e_H(F)$. But, since $H_2$ acts trivially on $B$, we have that $e_H(B)=e_{H_1}(B)\otimes T_2$, where $T_2$ is the trivial representation of $H_2$, and, analogously, $e_H(F)=T_1\otimes e_{H_1}(F)$, being $T_1$ the trivial representation of $H_1$.

    Therefore, 
    \begin{equation*}
        e_H(E)=e_{H_1}(B)\otimes e_{H_2}(F).
    \end{equation*}
    Given two representations $V_1$ and $V_2$ of $H_1$ and $H_2$ respectively, it happens that $V_1\otimes V_2$ is the trivial representation $T_1\otimes T_2$ of $H_1\times H_2$ if and only if $V_1=T_1$ and $V_2=T_2$, so we conclude that the coefficient of the trivial representation in $e_H(E)$ must be the product of the coefficient of $T_1$ in $e_{H_1}(B)$ and the coefficient of $T_2$ in $e_{H_2}(F)$, and, again by Remark \ref{coeftrivial}, these are $e(B/H_1)$ and $e(F/H_2)$ respectively.
\end{proof}

\begin{coro}\label{eSlambdaXn}
    Let $X$ be an algebraic variety and let $\lambda=(\lambda_1,\ldots,\lambda_l)$ be a partition of $n$. Suppose that $S_{\lambda}$ acts over $X^n$ by permutation of the corresponding coordinates. Then 
    \begin{equation*}
        e(X^n/S_{\lambda})=\prod_{i=1}^le(X^{\lambda_i}/S_{\lambda_i}).
    \end{equation*}
\end{coro}

\begin{proof}
    Applying Lemma \ref{Hactstrivial} to the projection $p_{\lambda_1}^n:X^n\to X^{\lambda_1}$, we get 
    \begin{equation*}
        e(X^n/S_{\lambda})=e(X^{\lambda_1}/S_{\lambda_1})e(X^{n-\lambda_1}/S_{\lambda'}),
    \end{equation*}
    where $\lambda'=(\lambda_2,\ldots,\lambda_l)$ is a partition of $n-\lambda_1$. Proceeding recursively on $X^{n-\lambda_1}/S_{\lambda'}$, we obtain the result.
\end{proof} 

\begin{coro}
    Let $X$ be an algebraic variety and let $\lambda=(\lambda_1,\ldots,\lambda_l)$ be a partition of $n$. Suppose that $S_{\lambda}$ acts over $C_n(X)$ by permutation of the corresponding coordinates. Then 
    \begin{equation*}
        e(C(X_n)/S_{\lambda})=\prod_{i=1}^le(C_{\lambda_i}(X-\{P_1,\ldots,P_{n_i}\})/S_{\lambda_i}),
    \end{equation*}
    where $n_i=\sum_{j=1}^{i-1}\lambda_j$.
\end{coro}

\begin{proof}
    The proof is analogous to the one of Corollary \ref{eSlambdaXn}, but taking into account that the fiber of $\pi_m^n:C_n(X)\to C_m(X)$ is $C_{n-m}(X-\{P_1,\ldots,P_m\})$.
\end{proof}

\begin{ejem}
    We are now going to use the results obtained in this paper to study the configuration space of orbits in a concrete case. Consider the action of $\mathbb{C}^*$ on $\mathrm{GL}_2=\mathrm{GL}_2(\mathbb{C})$. It is well known that the quotient of this variety for this action is just $\mathrm{GL}_2\sslash \mathbb{C}^*=\PGL_2$. Therefore, to study the equivariant $E$-polynomial of $C_n(\mathrm{GL}_2,\mathbb{C}^*)$, we just need to study the equivariant $E$-polynomial of $C_n(\PGL_2)$, so we can use the formulas obtained throughout this work. For example, for $n=2$ we want to compute
    \begin{equation*}
        e_{S_2}\left(C_n(\mathrm{GL}_2,\C^*)\right)=e_{S_2}\left(C_n(\PGL_2)\right)\otimes e_{S_2}\left((\C^*)^2\right).
    \end{equation*}
    So by formula (\ref{pmformula2}), we need to compute $e\left(C_2(\PGL_2)\right),e\left(C_2(\PGL_2)/S_2\right), e\left((\C^*)^2\right)$ and $e\left((\C^*)^2/S_2\right)$. But by Example \ref{geoejem} we know that 
    \begin{equation*}
        e\left(C_2(\PGL_2)\right)=q^6 - 2q^4 - q^3 + q^2 + q,
    \end{equation*}
    and, by Example \ref{CnquotSnejem}, we have
    \begin{equation*}
        e\left(C_2(\PGL_2)/S_2\right)=q^{6} - 4 q^{3} + 6.
    \end{equation*}
    It can be easily seen that $e\left((\C^*)^2\right)$ and $e\left((\C^*)^2/S_2\right)$ are $q^2-2q+1$ and $q^2-q$ respectively. So the equivariant $E$-polynomial is
    \begin{align*}
        e_{S_2}\left(C_2(\mathrm{GL}_2,\C^*)\right)&=\left[\left(q^6 - 2q^4 - q^3 + q^2 + q\right)T+\left(-2q^4+3q^3+q^2+q-6\right)N\right]\\
        &\qquad \otimes \left[\left(q^2-2q+1\right)T+\left(-q+1\right)N\right]\\
        &=\left(q^8 - 2q^7 - q^6 + 5q^5 - 4q^4 - q^2 + 8q - 6\right)T\\
        &\qquad +\left(-q^7 - q^6 + 9q^5 - 8q^4 - 7q^2 + 14*q - 6\right)N.
    \end{align*}
    Recall that by Remark \ref{coeftrivial}, this also gives us
    \begin{equation*}
        e\left(C_2(\mathrm{GL}_2,\C^*)\right)=q^8 - 3q^7 - 2q^6 + 14q^5 - 12q^4 - 8q^2 + 22q - 12,
    \end{equation*}
    and
    \begin{equation*}
        e\left(C_2(\mathrm{GL}_2,\C^*)/S_2\right)=q^8 - 2q^7 - q^6 + 5q^5 - 4q^4 - q^2 + 8q - 6.
    \end{equation*}
    For higher order, as usual, we need the help of a computer program to do all the computations, but we can still compute the equivariant $E$-polynomials. For $n=3$, for example, we get the formula
    \begin{align*}
        e_{S_3}\left(C_3(\mathrm{GL}_2,\C^*)\right)&=\left(q^{12} - q^{11} - q^{10} + q^{9} + 2 q^{8} - 4 q^{6} - q^{5} + q^{4} + 3 q^{3} + q^{2} - 2 q\right)T\\
        &+\left(-q^{11} + q^{9} + q^{8} - 2 q^{7} + q^{6} + 3 q^{5} - q^{4} - 3 q^{3} - q^{2} + 2 q\right)D\\
        &+\left(q^{10} + 2 q^{9} - q^{8} - 5 q^{7} - 3 q^{6} + 4 q^{5} + 5 q^{4} - q^{3} - 2 q^{2}\right)N,
    \end{align*}
    from which we get 
    \begin{equation*}
        e\left(C_3(\mathrm{GL}_2,\C^*)\right)=q^{12} - 3 \, q^{11} + 5 \, q^{9} + 3 \, q^{8} - 9 \, q^{7} - 5 \, q^{6} + 9 \, q^{5} + 4 \, q^{4} - 4 \, q^{3} - 3 \, q^{2} + 2 \, q
    \end{equation*}
    and
    \begin{equation*}
        e\left(C_3(\mathrm{GL}_2,\C^*)/S_3\right)=q^{12} - q^{11} - q^{10} + q^{9} + 2 q^{8} - 4 q^{6} - q^{5} + q^{4} + 3 q^{3} + q^{2} - 2 q.
    \end{equation*}
    As a last example, we do the case $n=4$, where our method gives us the equivariant $E$-polynomial
    \begin{align*}
        e_{S_4}\left(C_4(\mathrm{GL}_2,\C^*)\right)&=\left(q^{16} - q^{15} - q^{14} + q^{13} + 3 q^{12} - 2 q^{11} - 6 q^{10} + 2 q^{9}\right. \\
        &\left.\qquad +9 q^{8} + 6 q^{7} - 14 q^{6} - 11 q^{5} + 3 q^{4} + 13 q^{3} + 5 q^{2} - 8 q\right)V_{(4)}\\
        & +\left(-q^{15} + 4 q^{12} - 2 q^{11} - 5 q^{10} + 3 q^{8} + 5 q^{7} + q^{6}\right.\\
        &\left.\qquad - 5 q^{4} - 4 q^{3} + 2 q^{2} + 2 q\right)V_{(3,1)}\\
        &+\left(q^{13} - 2 q^{11} + q^{10} - 2 q^{9} + q^{8} - 2 q^{7} + 4 q^{6} + 4 q^{5} - 3 q^{4}\right.\\
        &\left.\qquad - 2 q^{3} - 3 q^{2} + 3 q\right)V_{(2,2)}\\
        &+\left(q^{14} + q^{13} - 2 q^{12} - 3 q^{11} + 4 q^{10} + 7 q^{9} - 2 q^{8} - 16 q^{7} - 5 q^{6} + 13 q^{5} \right.\\
        &\left.\qquad + 10 q^{4} - 3 q^{3} - 6 q^{2} + q\right)V_{(2,1,1)}\\
        &+\left(-2 q^{12} - 5 q^{11} - 2 q^{10} + 13 q^{9} + 13 q^{8}\right.\\
        &\left.\qquad - 3 q^{7} - 18 q^{6} - 8 q^{5} + 9 q^{4} + 4 q^{3} - q\right)V_{(1,1,1,1)},
    \end{align*}
    and the $E$-polynomials
    \begin{align*}
        e\left(C_4(\mathrm{GL}_2,\C^*)\right)&=q^{16} - 4 \, q^{15} + 2 \, q^{14} + 6 \, q^{13} + 7 \, q^{12} - 26 \, q^{11} - 9 \, q^{10} + 32 \, q^{9}\\
        &\qquad + 27 \, q^{8} - 34 \, q^{7} - 36 \, q^{6} + 28 \, q^{5} + 21 \, q^{4} - 8 \, q^{3} - 13 \, q^{2} + 6 \, q
    \end{align*}
    and
    \begin{align*}
        e\left(C_4(\mathrm{GL}_2,\C^*)/S_4\right)&=q^{16} - q^{15} - q^{14} + q^{13} + 3 q^{12} - 2 q^{11} - 6 q^{10} + 2 q^{9}\\
        &\qquad + 9 q^{8} + 6 q^{7} - 14 q^{6} - 11 q^{5} + 3 q^{4} + 13 q^{3} + 5 q^{2} - 8 q.
    \end{align*}
\end{ejem}

\begin{remark}
    From our method we can also obtain the $S_{\lambda}$-equivariant $E$-polynomial of $C_n(X,G)$ for any partition $\lambda=(\lambda_1,\ldots,\lambda_l)$ of $n$. Indeed, we have
    \begin{equation*}
        e_{S_{\lambda}}(C_n(X,G))=e_{S_\lambda}\left(C_n(X\sslash G)\right)\otimes e_{S_{\lambda}}(G^n),
    \end{equation*}
    but, as we saw in the proof of Lemma \ref{Hactstrivial}, 
    \begin{equation*}
        e_{S_{\lambda}}(G^n)=\bigotimes_{i=1}^le_{S_{\lambda_i}}(G^{\lambda_i}),
    \end{equation*}
    and
    \begin{equation*}
        e_{S_{\lambda}}(C_n(X\sslash G))=\bigotimes_{i=1}^le_{S_{\lambda_i}}\left(C_{\lambda_i}(X\sslash G-\{P_1,\ldots,P_{n_i}\})\right),
    \end{equation*}
    where $n_i=\sum_{j=1}^{i-1}\lambda_j$. So we have
    \begin{align*}
        e_{S_{\lambda}}(C_n(X,G))&=\bigotimes_{i=1}^le_{S_{\lambda_i}}\left(C_{\lambda_i}(X\sslash G-\{P_1,\ldots,P_{n_i}\}\right)\otimes \bigotimes_{i=1}^le_{S_{\lambda_i}}\left(G^{\lambda_i}\right)\\
        &=\bigotimes_{i=1}^l\left(e_{S_{\lambda_i}}\left(C_{\lambda_i}(X\sslash G-\{P_1,\ldots,P_{n_i}\}\right)\otimes e_{S_{\lambda_i}}\left(G^{\lambda_i}\right)\right)\\
        &=\bigotimes_{i=1}^le_{S_{\lambda_i}}(C_{\lambda_i}(X-\pi^{-1}(\{P_1,\ldots,P_{n_i}\}),G)).
    \end{align*}
    From this we can also deduce that 
    \begin{equation*}
        e\left(C_n(X,G)/S_{\lambda}\right)=\prod_{i=1}^le\left(C_{\lambda_i}\left(X-\pi^{-1}(\{P_1,\ldots,P_{n_i}\}),G\right)/S_{\lambda_i}\right).
    \end{equation*}
\end{remark}

\bibliographystyle{abbrv}
\bibliography{mybibliography}

\end{document}